\documentclass[a4paper,12pt]{article}
\usepackage{amsmath}
\usepackage{amssymb}
\usepackage{amscd}
\usepackage{amsthm}
\usepackage[totalwidth=17cm, totalheight=25cm]{geometry}
\usepackage{graphicx}

\newtheoremstyle{boldplain}
{9pt}
{9pt}
{\itshape}
{}
{\bfseries}
{.}
{.5em}
{\thmname{#1}\thmnumber{ #2}\thmnote{ (#3)}}%

\newtheoremstyle{bolddefinition}
{9pt}
{9pt}
{}
{}
{\bfseries}
{.}
{.5em}
{\thmname{#1}\thmnumber{ #2}\thmnote{ (#3)}}%

\theoremstyle{boldplain}

\newtheorem{lemma}[equation]{Lemma}
\newtheorem{proposition}[equation]{Proposition}

\newtheorem{theorem}[equation]{Theorem}

\theoremstyle{bolddefinition}

\newfont{\bigbf}{cmbx10 scaled\magstep1}
\normalbaselines
\numberwithin{equation}{section}

\newcommand{\pihalf}{\frac{\pi}{2}}

\DeclareMathOperator{\diam}{diam}

\DeclareMathOperator{\CAT}{CAT}

\begin{document}

\title{Convexity is a local property in $\CAT(\kappa)$ spaces}
\author{Carlos Ramos-Cuevas\footnote{
cramos@mathematik.uni-muenchen.de}}

\maketitle

\begin{abstract}
In this note we show that a connected, closed and
locally convex subset 
(with an extra assumption on the diameter with respect to
the induced length metric if $\kappa>0$)
of a $\CAT(\kappa)$ space is convex.
\end{abstract}

\section{Introduction}

The study of convex subsets of geodesic metric spaces is a very
natural and interesting geometric question. In this note we are 
interested in metric spaces with curvature bounded from above in
the sense of Alexandrov, that is, in $\CAT(\kappa)$ spaces.
Convex subsets play a special role in the study of the 
geometry of such spaces.

For instance, in \cite{KleinerLeeb:invconvex} 
Kleiner and Leeb study convex subsets of symmetric spaces
of noncompact type (which are examples of
simply connected 
Riemannian manifolds
of nonpositive curvature, in particular, $\CAT(0)$ spaces)
and their boundaries at infinity 
under the actions of certain groups of isometries.
These boundaries at infinity endowed with the Tits metric are 
examples of spherical buildings, which are in turn
$\CAT(1)$ spaces. 
Motivated by their analysis in \cite{KleinerLeeb:invconvex}
the authors ask if the circumradius of a convex
subset of a spherical building must be $\leq \pihalf$ or it
is itself a spherical building.
A weaker version of this question is Tits' Center Conjecture, 
which is concerned with
fixed points of groups of isometries acting on convex
subcomplexes of spherical buildings 
(see \cite{MuehlherrTits:centerconj}, \cite{LeebRamos-Cuevas:centerconj}
\cite{Ramos-Cuevas:centerconj}, \cite{MuehlherrWeiss:centerconj}
for a complete answer to this conjecture, 
we refer also to
\cite{BalserLytchak:centerbuild} for a related result).

In the case of $\CAT(0)$ spaces there is a {\em folklore} question
asking if the closed convex hull of a finite subset in a $\CAT(0)$ space
is compact. 
Clearly, the interesting case is when the space is not locally compact.
See e.g.\ mathoverflow.net/questions/6627/ for
a discussion on this problem.

If we regard unbounded subsets instead of finite subsets, we can ask now
how the asymptotic geometry is affected by taking the convex hull.
In \cite{HummelLangSchroeder:convexhulls} they give 
an example of a subset $S$ of a $\CAT(-1)$ space, such that 
the boundary at infinity of the convex hull $CH(S)$ is bigger than
the boundary at infinity of $S$. 
In contrast, they prove that if $S$ is the union
of finitely many convex subsets, then its convex hull has the same
boundary at infinity as $S$. This result is not true in general in the 
$\CAT(0)$ case (e.g.\ consider the Euclidean space ).

Consider the following related question. Which convex subsets $B$
of the ideal boundary (with the Tits metric) 
of a $\CAT(0)$ space $X$
can be realized as ideal boundaries of convex subsets of $X$?
In \cite{KleinerLeeb:invconvex} they address the case when $X$
is a symmetric space or a Euclidean building and $B\subset \partial_T X$
is a top-dimensional subbuilding. 
In \cite{Balser:convexrank1} Balser studies the case of Euclidean
buildings $X$ of type $A_2$ and
characterizes 0-dimensional convex subsets $B\subset \partial_T X$
arising as ideal boundaries of convex subsets of $X$. 

All this problems concern themselves in one way or another with
the constructions of convex subsets of $\CAT(\kappa)$ spaces.
It is, therefore, useful to have a tool to decide whether
a given subset is convex or not.
One can construct convex subsets of $\CAT(0)$ spaces by considering
points, geodesics, horoballs and taking tubular 
neighborhoods and intersections of them.
In \cite{Balser:convexrank1} the author also considers unions of these objects. 
In general
unions of convex sets will not be convex, hence, in order to verify that the
constructed subsets are convex, he proves and uses the following result.
Let $C\subset X$ be a closed connected subset of a $\CAT(0)$ space $X$ and suppose
that there is an $\varepsilon>0$ such that $C$ is $\varepsilon$-locally convex
(i.e.\ $B_p(\varepsilon)\cap C$ is convex for all $p\in C$), then $C$ is convex.
That is, one must verify the convexity of $C$ only locally.
One can actually drop the assumption that the $\varepsilon$ must be chosen 
uniformly in $C$ as is required in \cite{Balser:convexrank1}, 
this result is a consequence of the Cartan-Hadamard Theorem
\cite[Thm.\ 4.1]{BridsonHaefliger}.
Indeed, the assumption implies that $C$ is a locally $\CAT(0)$ space 
(in the restricted metric) and the
Cartan-Hadamard Theorem says that in this situation
the universal cover $\tilde C$ of $C$ is a 
$\CAT(0)$ space. Now we claim that $C=\tilde C$.
If there is a non null-homotopic loop in $C$, then its lift 
to $\tilde C$ is a non closed curve. Since $\tilde C$ is $\CAT(0)$, this
lift is homotopic to the unique geodesic segment joining its endpoints.
Projecting this homotopy
to $C$ shows that every non null-homotopic loop in $C$
is homotopic to a unique locally geodesic loop, but there are no 
locally geodesic loops in $X$, hence, $C$ is simply connected and a
$\CAT(0)$ space. In particular $C$ is a geodesic space and since 
geodesics in $X$ are unique we conclude that $C$ is convex.
(c.f.\ \cite[Sec.\ 24]{Gromov:catkspaces}, \cite{BuxWitzel:localconvex}.)

In the case of $\CAT(\kappa)$ spaces for $\kappa>0$ there is an analogous
result in \cite{Bowditch:locallycat1} to Cartan-Hadamard for 
locally compact spaces. It says that a locally compact, locally $\CAT(1)$
space is $\CAT(1)$ if and only if every loop of length $<2\pi$ can be homotoped
to a point by loops of length $<2\pi$ (this property being the analogous
of simply connected in the $\CAT(0)$ case).
One could use again 
this result as above to show that a closed, connected, locally convex
subset
(with an extra assumption on the diameter with respect
to the induced length metric, see Theorem~\ref{MainTheorem} below)
of a locally compact $\CAT(1)$ space is convex.
Another argument for the locally compact case using the Hopf-Rinow Theorem
can be found in \cite{BuxWitzel:localconvex}.
In this note we observe, 
that in order to show this result for $\CAT(\kappa)$
spaces, we do not need the full strength of the Cartan-Hadamard Theorem
and we can use instead the fact that our locally $\CAT(\kappa)$ space already
lives in an ambient $\CAT(\kappa)$ (and in particular, geodesic) space.
This allows us to give an unified proof for arbitrary $\kappa$ without the
assumption of local compactness. 

Let $D_\kappa$ be $\pi/\sqrt{\kappa}$ for $\kappa>0$ and 
$\infty$ for $\kappa\leq 0$ (see Section~\ref{Preliminaries}).
The main result of this note is:

\begin{theorem}\label{MainTheorem}
Let $C\subset X$ be a closed, connected subset of a $\CAT(\kappa)$ space 
$(X,d)$.
Suppose that for every point $p\in C$ there is an 
$\varepsilon=\varepsilon(p)>0$ such that
$B_p(\varepsilon)\cap C$ is convex,
that is, $C$ is locally convex.
Denote with $\ell$
the induced length metric in $C$. Then it holds:
\begin{enumerate}
\item  
If for two points $x,y\in C$ holds
$\ell(x,y)< D_\kappa$, then $xy\subset C$. In particular,
$\ell(x,y)\leq D_\kappa$ implies $\ell(x,y)=d(x,y)$.
\item
If the diameter $\diam_\ell(C)$ of $C$ with respect
to the length metric is $\leq D_\kappa$, then $C$ is a convex subset.
\item $(C,\ell)$ is a $\CAT(\kappa)$ space.
\end{enumerate}
\end{theorem}

Observe that in the case $\kappa\leq 0$ the assumption in (2) 
on the diameter is always satisfied and $C$ is always a convex subset.
If $\kappa>0$ then it is not true in general that $d(x,y)=\ell(x,y)$, 
consider e.g.\ a segment of length $\pi+\epsilon$ on the unit circle,
it is locally convex but not convex. Nevertheless, it is certainly 
$\CAT(1)$ with respect to the length metric (cf. with the result 
in Section~\ref{sec:spheres}).

Theorem~\ref{MainTheorem}(2) was first shown in the case of Euclidean spaces 
by Tietze and Nakajima 
in \cite{Tietze:convex}, \cite{Nakajima:convex}. For locally compact
Busemann spaces it has been proven by Papadopoulos in \cite[Sec.\ 8.3]{Papadopoulos}
(we note that our proof of Theorem~\ref{MainTheorem}(1-2) also
works unchanged for general Busemann spaces).
Bux and Witzel proved the cases of $\CAT(0)$ and locally compact $\CAT(1)$
spaces in \cite{BuxWitzel:localconvex}.

In Section~\ref{sec:spheres} we use Theorem~\ref{MainTheorem}
to show that in the case of subsets of dimension $\geq 2$ of 
spherical buildings, we do not need the assumption on the diameter, thus,
connected, closed, locally convex subsets are always convex. 
The condition on the dimension is 
necessary as explained in the example above.

I would like to thank S.\ Witzel for pointing me out some missing references
and drawing my attention to their preprint \cite{BuxWitzel:localconvex}.

\section{Preliminaries}\label{Preliminaries}

In this section we will recall some definitions and fix the notation.
We refer to \cite{BridsonHaefliger} for more 
information on $\CAT(\kappa)$ spaces.

\subsection{Metric spaces}
Let $(X,d)$ be a metric space. For $r>0$ and $x\in X$
we denote with $B_x(r)$ the {\em open metric ball}
of radius $r$ centered at $x$.

A {\em geodesic} segment is a curve whose
length realizes the distance between its endpoints.
If there is a unique geodesic between two points $x,y\in X$
we denote it with $xy$.
$X$ is said to be a {\em geodesic metric space} if there exists
geodesics between any two points and it is {\em d-geodesic} if
this is true for any two points at distance $<d$.

The {\em link} $\Sigma_xX$ at a point $x\in X$ is the {\em space
of directions} of geodesic segments in $X$ emanating from $x$
equipped with the angle metric. If $x\neq y$ and there
is a unique geodesic between them, we denote with 
$\overrightarrow{xy}\in\Sigma_xX$
its direction at $x$.

A {\em midpoint} between $x,y\in X$ is a point $z$, such that
$d(x,z) = d(z,y) = \frac{1}{2}d(x,y)$. If the midpoint is unique
we denote it with $m(x,y)$.

The {\em length metric} associated to $d$ is defined as follows: 
the distance between two points is the infimum of the lengths
(measured with respect to $d$) of rectifiable curves connecting them.

\subsection{$\CAT(\kappa)$ spaces}
Let $M_\kappa$ be the complete, simply connected 
Riemannian surface of constant sectional curvature $\kappa$
and let $D_\kappa$ denote its diameter. Then 
$D_\kappa=\pi/\sqrt{\kappa}$ for $\kappa>0$ and 
$D_\kappa=\infty$ for $\kappa\leq 0$.

Recall that a complete $D_k$-geodesic metric space $(X,d)$ 
is said to be $\CAT(\kappa)$ if all geodesic triangles of perimeter
$<2D_\kappa$ are not thicker than the corresponding comparison
triangles on the model space $M_\kappa$.

In a $\CAT(\kappa)$ space, points at distance $<D_\kappa$
are joined by a unique geodesic, and these geodesics vary
continuously with their endpoints. A local geodesic of
length $<D_\kappa$ is a geodesic.

We say that a subset $C$ of a $\CAT(\kappa)$ is {\em convex} if for
every two points in $C$ at distance $<D_\kappa$ the unique 
geodesic segment between them is contained in $C$.
A closed convex subset of a $\CAT(\kappa)$ space 
is itself $\CAT(\kappa)$.

\section{Proof of the main result}

In this section all geometric objects like geodesics, metric balls, etc.\
will be taken with respect to the metric $d$ if not said otherwise.

First notice that the induced length metric $\ell$ in $C$ is finite. 
Indeed, the set of points that can be connected to a given point 
with a rectifiable curve in $C$ is clearly open and closed in $C$, 
because of the local convexity of $C$.

Let $x\in C$ and define $C_x$ as the subset of $C$ of points $y$
such that $d(x,y)<D_\kappa$ and the geodesic segment
$xy$ is contained in $C$. In particular $\ell(x,y)=d(x,y)$ for all
$y\in C_x$.

\begin{lemma}\label{mainlemma}
 $C_x$ is open in $C$.
\end{lemma}
\begin{proof}
 Let $y\in C_x$ and write $D:=d(x,y)<D_\kappa$.
By local convexity, we can choose an $\varepsilon$ with
$D_\kappa/2>\varepsilon>0$ and
such that $B_{y'}(\varepsilon)\cap C$ is convex for all $y'\in xy$.
We follow the proof of \cite[Lemma~4.3(1)]{BridsonHaefliger}
and modify it to work in our setting. The proof is some kind of induction.
Consider the following statement in the real number $d\geq 0$:

\medskip
{\em There exists a positive number $\epsilon_d>0$ depending only on 
$d\leq D$ such that if $x',y'\in xy\subset C$ and $d(x',y')\leq d$,
then for any $\bar x\in B_{x'}(\epsilon_d)\cap C$
and $\bar y\in B_{y'}(\epsilon_d)\cap C$
holds $d(\bar x,\bar y)<D_\kappa$ and the geodesic
$\bar x\bar y$ lies in $C$.}

\medskip
Notice that if the statement is true for some $d$, then it is true
for any $d'<d$. If it is true for $D$ then we are done: 
the open set
$B_y(\epsilon_D)\cap C$ is then contained in $C_x$.

{\em Induction basis:} 
If $d\leq \varepsilon/2$ then the statement is true: Take 
$\epsilon_d:=\varepsilon/2$. Then any $\bar x, \bar y$ as in the statement
lie in $B_{x'}(\varepsilon)\cap C$, which is convex by assumption.
Further, $d(\bar x, \bar y)\leq 3\varepsilon/2<3D_\kappa/4$.

{\em Induction step:}
Now assume that the statement is true for some $d\leq D$, then
we want to show, that it is true for $d':=\min \{3d/2,D\}$.
Let $\delta:=D_\kappa-D>0$.

By triangle comparison and Proposition~\ref{prop:sphericaltriangle},
there exists a constant $K=K(D)<1$ such that if $x,y,z\in X$
are the vertices of a geodesic triangle in $X$ with perimeter $<2D_\kappa$
and $d(x,y),d(x,z)\leq \frac{2}{3}(D+\delta/2)< \frac{2}{3}D_\kappa$
then $d(m(x,y),m(x,z))\leq Kd(y,z)$.

Let $\epsilon_{d'}:=(1-K)\min\{\epsilon_d,\delta/3\}$.
If $x',y',\bar x,\bar y\in C$ are points as in the statement, then 
$d(\bar x,\bar y)\leq d(x',y') + 2\epsilon_{d'}
\leq D + 2(1-K)\delta/3 < D + 2\delta/3 < D_\kappa$.
This shows the first conclusion of the statement.

Let $a_0,b_0\in x'y'\subset xy$ be two points
such that $d(x',a_0)=d(a_0,b_0)=d(b_0,y')$, that is, $a_0,b_0$ 
cut the segment $x'y'$ in thirds.
Since $d(x',b_0)=\frac{2}{3}d(x',y')\leq \frac{2}{3}d'\leq d$
and $\bar x\in B_{x'}(\epsilon_{d'})\cap C\subset B_{x'}(\epsilon_d)\cap C$,
by the induction hypothesis it follows that $\bar xb_0\subset C$.
Analogously, $a_0\bar y\subset C$. 
We define inductively $a_{n+1}:=m(\bar x, b_n)$ and 
$b_{n+1}:=m(\bar y, a_n)$ (see Figure~\ref{fig1}). 
That they are well defined 
(i.e.\ $d(\bar x, b_n), d(\bar y, a_n) < D_\kappa$) is a consequence 
of (iii) below. We want to prove the following.
\begin{enumerate}
 \item[(i)] $d(a_{n-1},a_n),d(b_{n-1},b_n)\leq K^n\epsilon_{d'}$.
 \item[(ii)] $\bar x b_n, \bar y a_n \subset C$.
 \item[(iii)] $d(\bar x, b_n),d(\bar y, a_n)\leq \frac{2}{3}(D+\delta/2)$.
\end{enumerate}
\begin{figure}\begin{center}
\includegraphics[scale=0.6]{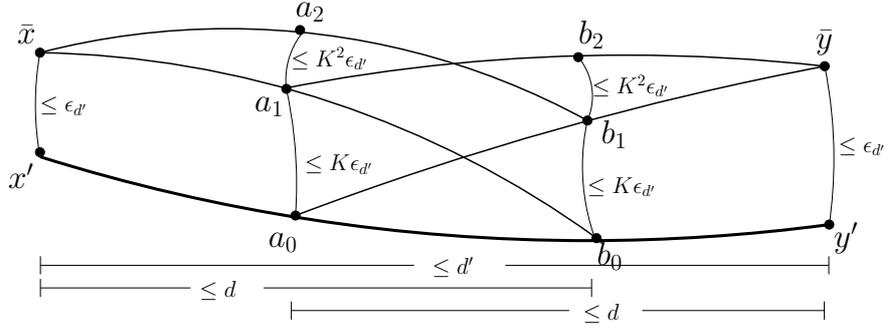}
\caption{Construction of the sequences $\{a_n\},\{b_n\}$}
\label{fig1}
\end{center}
\end{figure}
This would finish the proof of the induction step. Indeed, 
by (i) the sequences $\{a_n\},\{b_n\}$ are Cauchy and
because
$X$ is complete and $C$ is closed, there are points $a,b\in C$
with $a_n\rightarrow a$ and $b_n\rightarrow b$.
The geodesics $\bar x b_n, \bar y a_n \subset C$ converge to geodesics
$\bar x b, \bar y a \subset C$, because geodesics of length $<D_\kappa$
in $CAT(\kappa)$ spaces vary continuously with their endpoints.
The midpoint of the geodesic $\bar x b$ is $a$
and the midpoint of the geodesic $a\bar y$ is $b$, hence,
the union of $\bar x b$ and $a\bar y$ is a local geodesic between
$\bar x$ and $\bar y$. This local geodesic must coincide with
$\bar x\bar y$ because 
it has length $d(\bar x, a)+d(a,\bar y)=
\frac{1}{2}d(\bar x, b)+d(a,\bar y)\leq 
\frac{1}{3}(D+\delta/2)+\frac{2}{3}(D+\delta/2)<D_\kappa$, by (iii).
It follows that 
$\bar x\bar y\subset C$.

(ii) follows from (i) and the induction hypothesis for $d$:
We have that
$d(x',b_0),\linebreak d(a_0,y')\leq \frac{2}{3}d'\leq d$ and 
$\bar x\in B_{x'}(\epsilon_{d'})\cap C\subset B_{x'}(\epsilon_d)\cap C$,
respectively,
$\bar y\in B_{y'}(\epsilon_{d'})\cap C\subset B_{y'}(\epsilon_d)\cap C$.
By (i), it holds 
$d(a_0,a_n)\leq \sum_{i=1}^n K^i\epsilon_{d'} < \epsilon_{d'}/(1-K)
\leq \epsilon_d$. 
Analogously, $d(b_0,b_n) < \epsilon_d$.
Thus, if $a_n,b_n\in C$ then by the induction hypothesis,
$\bar x b_n, \bar y a_n \subset C$, which in turn implies
that $a_{n+1},b_{n+1}\in C$. Now (ii) follows inductively.

It remains to prove (i) and (iii). 
Suppose (i) is true for all $n\leq m$. 
Then for $n\leq m$ holds
$d(\bar y, a_n)\leq d(\bar y, y')+d(y',a_0)+
\sum\limits_{i=1}^n d(a_{i-1},a_i)\leq \epsilon_{d'}+\frac{2}{3}d'
+\sum\limits_{i=1}^n K^i\epsilon_{d'} 
\leq \frac{2}{3}d' + \frac{\epsilon_{d'}}{1-K}
\leq \frac{2}{3}(D+\frac{\delta}{2})$.
Analogously, $d(\bar x, b_n)\leq \frac{2}{3}(D+\frac{\delta}{2})$.
Thus, (iii) is true for all $n\leq m$.
Consider the triangle $(\bar y, a_{m-1},a_m)$.
Then, by the definition of the constant $K$ above 
(given by Proposition~\ref{prop:sphericaltriangle}), 
it follows that
$d(b_m,b_{m+1})=d(m(\bar y, a_{m-1}),m(\bar y, a_m)) 
\leq Kd(a_{m-1},a_m)\leq K^{m+1}\epsilon_{d'}$.
Analogously, $d(a_m,a_{m+1})\leq K^{m+1}\epsilon_{d'}$.
Hence, (i) is true for all $n\leq m+1$.
Now (i) and (iii) follow again inductively.
\end{proof}

Let now $y\in C$ such that $\ell(x,y)< D_\kappa$. 
Let $\gamma$ be a curve in $C$ connecting $x$ and $y$
of length $L(\gamma)$ with
$\ell(x,y) \leq L(\gamma) <D_\kappa$.
By Lemma~\ref{mainlemma}, the set
 $C_x\cap\gamma$ is open in $\gamma$.
It is also closed because $L(\gamma)<D_\kappa$ and geodesics 
of length $<D_\kappa$ 
in $\CAT(\kappa)$ spaces vary continuously with their endpoints.
Hence, $y\in C_x$ and $xy\subset C$. This proves the first assertion
of (1).

If $\ell(x,y)=D_\kappa$, then we can choose points 
$y_n\in C$ with $\ell(x,y_n)<D_\kappa$ and $\ell(y_n,y)\rightarrow 0$.
The second assertion of (1) follows by 
the first part and continuity.
(2) is a direct consequence of (1).

To show (3), first notice that by (1), $(C,\ell)$ is $D_\kappa$-geodesic.
Let $x_0,x_1,x_2\in C$ be the vertices of a geodesic 
(with respect to $\ell$)
triangle $\Delta$ of perimeter $<2D_\kappa$. In particular, each side
of the triangle has length $<D_\kappa$. Then by (1) it follows
that $\ell(x_i,x_j)=d(x_i,x_j)$ and $x_ix_j\subset C$. Hence
$\Delta$ is also a geodesic triangle with respect to $d$.
Let $a,b$ be two points on the sides of $\Delta$, they
split the triangle in two curves in $C$ and at least
one of them has length
$<D_\kappa$. This implies that $\ell(a,b)<D_\kappa$ and therefore
$\ell(a,b)=d(a,b)$. That is, $(\Delta,\ell)$ is isometric to
$(\Delta,d)$. The $\CAT(\kappa)$ property for $(C,\ell)$
follows from the $\CAT(\kappa)$ property of $(X,d)$.
\qed

\section{Locally convex subsets of spherical buildings}
\label{sec:spheres}

We apply now Theorem~\ref{MainTheorem} in the case of spherical buildings
and obtain following result.

\begin{theorem}\label{prop:sphere}
 Let $C$ be a connected, closed, locally convex subset of a 
spherical building
$B$. Suppose that $\dim(C)\geq 2$, then $C$ is convex.
\end{theorem}
\begin{proof}
Suppose there are two points $x,y\in C$ with $\ell(x,y)>\pi$ 
and let $\gamma$ be a curve in $C$ connecting them. 
There must be a point $\hat x\neq y$ in $\gamma$ with $\ell(x,\hat x)=\pi$ and 
therefore $d(x,\hat x)=\pi$ by Theorem~\ref{MainTheorem}. That is, the curve
$\gamma$ contains antipodes of $x$.
The antipodes of $x$ have the same type (i.e.\ the same image in the model Weyl chamber
under the natural projection). Recall that the distances between points of the same type 
in a spherical building
have finitely many different possible values. This implies that $\gamma$ contains
only finitely many antipodes $\hat x_1,\dots,\hat x_k$
of $x$. We can also assume that $\gamma$ meets each one of these
antipodes only once.

The next step is to observe that since $\dim(C)\geq 2$ we can change the curve $\gamma$
to avoid each one of the $\hat x_1,\dots,\hat x_k$ and obtain a curve 
connecting $x$ and $y$ that does not meet
any antipode of $x$. 
Let us make this observation more precise.
Let $r>0$ be such that $\overline{B_{\hat x_1}(r)}$ does not contain any antipode of $x$ and 
$\overline{B_{\hat x_1}(r)}\cap C$ is convex.
Then $(\overline{B_{\hat x_1}(r)}\cap C)\setminus \{\hat x_1\} $ is connected.
Let $x_1, y_1\in C$ be te first, respectively the last, point of $\gamma$
in $\overline{B_{\hat x_1}(r)}\cap C$. Hence we can replace the part of $\gamma$
between $x_1$ and $y_1$ with a curve in 
$(\overline{B_{\hat x_1}(r)}\cap C)\setminus \{\hat x_1\}$.
Repeating this procedure, we obtain a curve in $C$ connecting $x$ and $y$
avoiding any antipodes of $x$. This contradicts our first observation
at the beginning of the proof. It follows that $\diam_\ell(C)\leq\pi$ and 
Theorem~\ref{MainTheorem} implies that $C$ is convex.
\end{proof}

\appendix
\section{A computation in spherical trigonometry}

\begin{proposition}\label{prop:sphericaltriangle}
Consider a triangle in the 2-dimensional unit sphere $S^2$ with 
vertices $x,y,z$ and sides lengths $a=d(y,z),b=d(x,z),c=d(x,y)$.
Let $y',z'$ be the midpoints of the segments 
$xy$ and $xz$ respectively
and let $c'=d(y',z')$.
Suppose that $a,b\leq C < \frac{2\pi}{3}$.
Then there exists a constant $K=K(C)< 1$, such that $c'\leq Kc$.
\end{proposition}
\begin{proof}
Let $\theta:=\angle_x(y,z)$ be the angle in $x$, 
then the spherical law of cosines for the sides $c$ and $c'$ imply
\begin{align*}
(\cos\frac{a}{2}\cos\frac{b}{2}) \cos c'  &=
(\cos\frac{a}{2}\cos\frac{b}{2})(\cos\frac{a}{2}\cos\frac{b}{2} + 
      \sin\frac{ a}{2}\sin\frac{ b}{2}\cos\theta)\\
&=    \cos^2\frac{a}{2}\cos^2\frac{b}{2} + 
      \frac{\sin a}{2}\frac{\sin b}{2}\cos\theta \\
&= \cos^2\frac{a}{2}\cos^2\frac{b}{2} + \frac{1}{4}(\cos c-\cos a\cos b) \\
&= \frac{1}{4}(4\cos^2\frac{a}{2}\cos^2\frac{b}{2} +\cos c -
    (2\cos^2\frac{a}{2} - 1)(2\cos^2\frac{b}{2} - 1)) \\
&= \frac{1}{4}(\cos c -1 +2(\cos^2\frac{a}{2} + \cos^2\frac{b}{2})).
\end{align*}
Suppose first that $a=b$ and write $\alpha := \cos\frac{a}{2}$. Then we have
$4\alpha^2(\cos c' -1) = \cos c -1$, 
or equivalently, $\sin^2\frac{c'}{2} = \frac{1}{4\alpha^2}\sin^2\frac{c}{2}$.
Notice that $\frac{1}{4\alpha^2} \leq \frac{1}{4\cos^2 C/2} <1$.
Since $c,c' < \pi$, this implies that there is a constant $K=K(C)< 1$
such that $c'\leq Kc$ proving the proposition in this case.

Now we may assume w.l.o.g.\ that $b\leq a$.
Let $\beta := \cos\frac{b}{2} \geq \cos\frac{a}{2} = \alpha > \frac{1}{2}$.
We fix $a,c$ and consider $c'$ as a function of $b$, or equivalently,
as a function $f(\beta):=c'$ of $\beta$.
We can derive the equation above with respect to $\beta$ and obtain
$$
\alpha(\cos f(\beta)-\beta f'(\beta) \sin f(\beta)) = \beta.
$$
Hence $f'(\beta)\leq 0$ if and only if $\alpha\cos f(\beta)-\beta\leq 0$. 
The latter can be seen easily using again the equation above
for $\cos c'$:
\begin{align*}
 \alpha\cos f(\beta)-\beta &= \frac{4\alpha\beta\cos f(\beta)-4\beta^2}{4\beta}
    = \frac{\cos c-1 + 2(\alpha^2+\beta^2)-4\beta^2}{4\beta}\\
  &= \frac{(\cos c -1) + 2(\alpha^2 - \beta^2)}{4\beta}\leq 0.
\end{align*}
Thus, $f$ is monotone non-increasing and in particular,
$c' = f(\beta) \leq f(\alpha) \leq Kc$ by the first case considered above.
\end{proof}

\bibliography{../MyBibliography}
\bibliographystyle{alpha}
\end{document}